\documentclass[12pt,a4paper]{article}
 
\usepackage[T1]{fontenc}
\usepackage[utf8]{inputenc}
\usepackage{fullpage}
\usepackage{amsfonts,amsmath,amssymb,amsthm,mathtools,dsfont,
  xcolor,hyperref}
\usepackage[shortlabels]{enumitem}
\usepackage{authblk}

\hypersetup{colorlinks,linkcolor={red!50!black},
	citecolor={blue!50!black}}

\newtheorem{theorem}{Theorem}[section]
\newtheorem{lemma}[theorem]{Lemma}

\theoremstyle{definition}

\theoremstyle{remark}
\newtheorem{remark}[theorem]{Remark}

\newcommand{\eps}{\varepsilon}
\newcommand{\one}{\mathds 1}
\renewcommand{\P}{\mathds P}
\newcommand{\E}{\mathds E}
\newcommand{\R}{\mathds R}
\newcommand{\sR}{\mathcal R}
\newcommand{\N}{\mathds N}
\renewcommand{\O}{\mathcal O}
\newcommand{\hs}{\hspace{1pt}}
\newcommand{\e}{\mathrm e}
\renewcommand{\d}{\,\mathrm d}
\newcommand{\B}{\mathcal B}
\DeclareMathOperator{\Var}{Var}

\newcommand\smallO{
  \mathchoice
  {{\scriptstyle\mathcal{O}}}
  {{\scriptstyle\mathcal{O}}}
  {{\scriptscriptstyle\mathcal{O}}}
  {\scalebox{.7}{$\scriptscriptstyle\mathcal{O}$}}}

\numberwithin{equation}{section}
\usepackage[most]{tcolorbox} % for colored boxes

\newtcolorbox{query}{
  colback=white,         % white background
  colframe=black,         % black border
  boxrule=0.8pt,          % thickness of the border
  arc=0pt,                % square corners
  left=4pt, right=4pt, top=4pt, bottom=4pt, % padding
}

\title{Pareto points in growing dimensions}

\author[1,2]{Andrii Ilienko\thanks{The first author is supported by the Swiss National Science
  Foundation under grant no. 229505. \\ \hspace*{\parindent}\phantom{$\ast$} Email: \href{mailto: andrii.ilienko@unibe.ch}{\nolinkurl{andrii.ilienko@unibe.ch}.}}}
\author[1]{Bochen Jin\thanks{Email: \href{mailto:bochen.jin@unibe.ch}{\nolinkurl{bochen.jin@unibe.ch}.}}}

\affil[1]{Institute of Mathematical Statistics and Actuarial Science,\newline
  University of Bern, Switzerland}
\affil[2]{Igor Sikorsky Kyiv Polytechnic Institute, Ukraine}

\date{}

\begin{document}
\maketitle

\begin{abstract}
  We consider $n$ independent random points uniformly distributed in
  the $d_n$-dimensional unit cube and study Pareto points, that is,
  points that do not coordinatewise dominate any other point. We
  identify the critical growth rate of $d_n$ at which a phase transition
  occurs: below this threshold, the number of non-Pareto points
  diverges in probability, whereas above it there are asymptotically
  no such points. At criticality, the number of non-Pareto points
  converges in distribution to a Poisson random variable. We further
  describe their asymptotic spatial distribution in terms of
  convergence of random point measures.
  
  We also investigate points that dominate exactly $r$ other points and
  establish analogous phase transitions. For $r=1$, the critical
  dimension is the same as for non-Pareto points, whereas for every
  fixed $r\geq 2$ it is different, but, surprisingly, common to all
  such $r$.


\end{abstract}
\section{Introduction and overview}

Let $X_i$, $i=1,\ldots,n$, be distinct points in $\R^d$ with coordinates $X_i^{k}$, $k=1,\ldots,d$. We say that $X_i$ dominates $X_j$ and write $X_i\succeq X_j$ if $X_i^{k}\ge X_j^{k}$ for all $k$; $X_i\preceq X_j$ stands for the reverse relation. A point $X_i$ is called Pareto-minimal (or simply Pareto) if it does not dominate any other point, i.e.\ if there is no $j\ne i$ such that $X_i\succeq X_j$.
Pareto points represent the set of efficient trade-offs in multiobjective optimization and, more broadly, provide a basic model for extracting extremal frontiers from multivariate data. From a geometric viewpoint, Pareto points constitute the lower boundary of the point cloud in the direction $(-1,\ldots,-1)$ with respect to the coordinatewise order.

We study Pareto points among i.i.d.\ random points $X_1,\ldots,X_n$ drawn uniformly from the $d$-dimensional unit cube. Note that if the coordinates of each point are independent with continuous marginal distributions, the model can be reduced to the uniform case by applying the coordinatewise quantile transform. Since the seminal work \cite{BS66}, Pareto points of random samples have been extensively studied. Note that in \cite{BS66} and in the other works cited below, the authors consider Pareto-maximal points, i.e., points not dominated by any other point. For convenience, we work with the dual notion. Clearly, the two formulations are equivalent from a probabilistic viewpoint.

Let $K_{n,d}$ denote the number of Pareto points in a sample of size $n$ drawn uniformly from $[0,1]^d$. It follows from the results in \cite{BS66} and \cite{BCHL98} that, as $n\to\infty$ with fixed $d\ge2$,
\begin{gather}
	\label{eq:E}
	\E K_{n,d}\sim\frac1{(d-1)!}(\log n)^{d-1},\\
	\Var K_{n,d}\sim\biggl(\frac1{(d-1)!}+c_d\biggr)(\log n)^{d-1}.
	\label{eq:Var}
\end{gather}
The constants $c_d$ admit an explicit triple-series representation. Moreover, the first few constants are simple multiples of values of the Riemann zeta function: $c_2=0$, $c_3=\zeta(2)$, $c_4=2\zeta(3)$, $c_5=\frac{33}{16}\zeta(4)$.
In particular, \eqref{eq:E} and \eqref{eq:Var} immediately imply a law
of large numbers for $K_{n,d}$. A quantitative central limit theorem
for $K_{n,d}$ was established in \cite{BDHT05} for independent points
$X_i$, and in \cite{BM22} for samples drawn from a homogeneous Poisson
point process. For related CLTs in more general settings, see also
\cite{B00,BX01}.

Equation \eqref{eq:E} shows that the order of growth of $\E K_{n,d}$ in $n$ increases with the dimension. This is in line with intuition: adding dimensions makes dominance more restrictive, so fewer points dominate others and more become Pareto. This naturally suggests letting the dimension $d=d_n$ grow with $n$ in the hope of revealing new phenomena. To the best of our knowledge, relatively little is known about Pareto points in this regime: first-order asymptotics for $\E K_{n,d_n}$ were obtained in \cite{H04}, and, more recently, \cite{JZ23} studied a phase transition for the probability $p_{n,d_n}$ that a point (say, $X_1$) is Pareto. Specifically, it is shown in \cite{JZ23} that if $d_n$ grows at a slower (resp.\ faster) rate than $\log n$, then $p_{n,d_n}\to0$ (resp.\ $p_{n,d_n}\to1$) as $n\to\infty$.

We begin by considering two extreme regimes. If $d_n=1$, then there is clearly only one Pareto point, namely the minimum of the sample, so that the number of non-Pareto points is $n-K_{n,1}=n-1\to\infty$. On the other hand, say for $d_n=n$,
\[\{n-K_{n,n}>0\}\subseteq\bigcup_{1\le i\ne j\le n}\{X_i\preceq X_j\},\]
and hence
\[\P\{n-K_{n,n}>0\}\le n(n-1)\P\{X_1\preceq X_2\}=\frac{n(n-1)}{2^n}\to0.\]
Thus, $n-K_{n,n}\xrightarrow{p}0$. These examples suggest that, in some intermediate growth regime of $d_n$, the number of non-Pareto points should converge to a (random) finite limit. Our first result, Theorem \ref{thm:1}, shows in particular that this phase transition occurs when \[d_n=d_n^{\hs\star}\vcentcolon=\beta\log n+\O(1)\]
with $\beta=\frac2{\log2}\approx2.885$, and that the limiting law is Poisson. Note that this critical dimension is greater than the $\log n$ threshold for $p_{n,d_n}$ in \cite{JZ23}. This is not particularly surprising: even if each point is Pareto with probability tending to $1$, the increasing sample size may still produce an unbounded number of non-Pareto points.

The Poisson nature of the limit allows for a finer description: it enables us to study not only the number of non-Pareto points in this regime, but also their spatial distribution. In Theorem \ref{thm:2} we show that the random point measures obtained by projecting the non-Pareto points onto any fixed finite-dimensional coordinate subspace converge in distribution to an inhomogeneous Poisson point measure on that subspace.

We then turn to the finer structure of the non-Pareto points. For $r=1,\ldots,n-1$, let $K_{n,d}^{(r)}$ denote the number of points among $X_1,\ldots,X_n$ that dominate exactly $r$ others. In particular, $n-\sum_{r=1}^{n-1}K_{n,d}^{(r)}=K_{n,d}$ is the number of Pareto points.
The asymptotic behavior of $K_{n,d}^{(r)}$ for fixed $d$ has also been studied in the literature, where it corresponds to the $(r+1)$st layer. For instance, \cite{BS66} shows that if $r$ is fixed as well, then $\E K_{n,d}^{(r)}$ shares the same asymptotics \eqref{eq:E} as $\E K_{n,d}$.
It follows from Theorem \ref{thm:1} that, in the critical regime $d_n=d_n^{\hs\star}$, we have
\[\sum_{r=2}^{n-1} K_{n,d_n}^{(r)}\xrightarrow{p}0,\]
that is, asymptotically all points dominate at most one other point. It would be natural to conjecture that there exists a family of threshold sequences $\bigl(d_n^{(r)}\bigr)_{n\in\N}$, decreasing in $r$, with $d_n^{(1)}=d_n^{\hs\star}=\beta\log n+\O(1)$, such that $K_{n,d_n}^{(r)}\xrightarrow{p}\infty$ (resp.\ $\xrightarrow{p}0$) whenever $d_n$ grows at a slower (resp.\ faster) rate than $d_n^{(r)}$. Surprisingly, this is not the case.

It turns out that all $r\ge2$ share the same threshold:
\begin{equation}
	\label{eq:elog}
	d_n^{(r)}=d_n^{\hs\star\hspace{-0.5pt}\star}\vcentcolon=\e\log n-\frac12\log\log n+\O(1).
\end{equation}
Proving a Poisson limit for $K_{n,d_n}^{(r)}$ in this regime appears to be a highly nontrivial problem. In particular, the Stein--Chen Poisson approximation used in the proofs of Theorems \ref{thm:1} and \ref{thm:2} breaks down here due to substantially stronger dependencies. Accordingly, we leave questions concerning the limiting distribution --- let alone the convergence of the associated projected point measures --- beyond the scope of this paper. Instead, in Theorem~3 we establish the universality of the threshold \eqref{eq:elog} at the level of expectations. More precisely, for every $r\ge2$ we show that $\E K_{n,d_n}^{(r)}\to\infty$ or $\to0$ depending on whether $d_n$ grows slower or faster than $d_n^{\hs\star\hspace{-0.5pt}\star}$, and identify the limit of $\E K_{n,d_n}^{(r)}$ in the critical regime.

Finally, we note that the study of Pareto points in growing dimensions is of interest not only in its own right, but also from the perspective of stochastic geometry. A classical object in this area is the convex hull of i.i.d.\ random points. Pareto points, which form the lower boundary of the sample, are closely related to the vertices of the convex hull, although the latter are more difficult to analyze. This connection also arises at the algorithmic level; see \cite{BKST78,D80}. Recently, there has been progress on convex hulls of i.i.d.\ random points in high dimensions, in particular for the uniform model (and, more generally, for samples from multivariate beta and beta-prime distributions). In this growing-dimensional regime, threshold phenomena have been identified, and asymptotic formulas derived for the expectations of the volume, intrinsic volumes, and the number of vertices; see, e.g., \cite{BCGTT19,BKT21}. As we demonstrate in this paper, much finer results can be obtained for Pareto points. It is plausible that similar phenomena may eventually be established for the vertices of convex hulls as well.

\section{Main results}

We recall the notation. For i.i.d.\ points $X_i$, $i=1,\ldots,n$, distributed uniformly in $[0,1]^d$, let $K_{n,d}$ denote the number of Pareto points among them, and $K_{n,d}^{(r)}$, $r=1,\ldots,n-1$, the number of points that dominate exactly $r$ others. In particular,
$K_{n,d}+\sum_{r=1}^{n-1}K_{n,d}^{(r)}=n$.
Our first result establishes the critical dimension and the structure of the phase transition for $n-K_{n,d}$, the number of non-Pareto points.

\begin{theorem}\
  \label{thm:1}
  \begin{enumerate}[(i)]
  \item If $d_n-\frac 2{\log 2}\log n\to-\infty$, then $n-K_{n,d_n}\xrightarrow p\infty$.
  \item If $d_n-\frac 2{\log 2}\log n\to c\in\R$, then $n-K_{n,d_n}\xrightarrow dZ_c\sim\mathsf{Pois}(2^{-c})$, where $\mathsf{Pois}$ stands for the Poisson distribution. Moreover, in this case,
    \[\E(n-K_{n,d_n})\to\E Z_c=2^{-c},\qquad\Var(n-K_{n,d_n})\to\Var Z_c=2^{-c},\qquad\sum_{r=2}^{n-1} K_{n,d_n}^{(r)}\xrightarrow{p}0.\]
  \item If $d_n-\frac 2{\log 2}\log n\to\infty$, then $n-K_{n,d_n}\xrightarrow p0$.
  \end{enumerate}
\end{theorem}

Since $d_n$ is integer-valued, the condition in (ii) cannot hold
along all $n$ for any fixed $c\in\R$. Therefore, this and all later results
are understood in the subsequential sense, that is, along subsequences
for which the corresponding assumptions are satisfied. We shall not
mention this explicitly thereafter.

To formulate the next theorem, recall that the vague topology on the space of locally finite measures on a Polish space $S$ is generated by the integration maps 
$\nu\mapsto\int_S f\d\nu$ for all bounded continuous functions 
$f$ with bounded support; see, e.g., Section 3.4 in \cite{R87} 
or Chapter 4 in \cite{K17} for a general exposition. 
Denote by $\xrightarrow{\,vd\,}$ the distributional convergence 
of random point measures with respect to the vague topology.

Fix $m\in\N$ and integers $1\le k_1<\cdots<k_m$. For each $n$ with $d_n\ge k_m$, let $\pi_m\colon\R^{d_n}\to\R^m$ denote the canonical projection onto the subspace spanned by the $k_1$th,\ldots,$k_m$th coordinate axes:
\begin{equation}
\label{eq:pi}
\pi_m\bigl(x^{1},\ldots,x^{d_n}\bigr)
=\bigl(x^{k_1},\ldots,x^{k_m}\bigr).
\end{equation}
Throughout, we write $\pi_m$ without indicating the dependence on $k_1,\ldots,k_m$, since only the number $m$ of selected coordinates will matter. For $n$ with $d_n\ge k_m$, consider the random point measure $\Xi_n^{m\vphantom{\int}}$ on $[0,1]^m$ given by
\[\Xi_n^{m\vphantom{\int}}=\sum_{i=1}^n \delta_{\pi_m X_i}
\one\{X_i\ \text{is non-Pareto}\},\]
where $\delta_a$ stands for the Dirac measure at $a$. In fact, $\Xi_n^{m\vphantom{\int}}$ is the pushforward of the point measure of all non-Pareto points under the canonical projection $\pi_m$.

For $x\in\R^m$, let $|x|$ denote the product of its coordinates. The next result establishes the limiting spatial distribution of the projections of non-Pareto points in the critical regime.

\begin{theorem}
	\label{thm:2}
	Let $m\ge1$. If $d_n-\frac 2{\log 2}\log n\to c\in\R$, then $\Xi_n^{m\vphantom{\int}}\xrightarrow{\,vd\,}\Xi^{m\vphantom{\int}}$, where $\Xi^{m\vphantom{\int}}$ is a Poisson point measure on $[0,1]^m$ with intensity $\lambda^m$ given by
	\[\d\lambda^m=2^{m-c}|x|\d x.\]
\end{theorem}

Our final result establishes the critical dimension for the phase transition in $K_{n,d}^{(r)}$, $r\ge2$, and the limiting expectations in this regime.

\begin{theorem}\
  \label{thm:3}	
  \begin{enumerate}[(i)]
  \item If $d_n-(\e\log n-\frac12\log\log n)\to-\infty$
    and $d_n\neq 1$,
    then $\E K_{n,d_n}^{(r)}\to\infty$ for any $r\ge2$.
  \item If $d_n-(\e\log n-\frac12\log\log n)\to c\in\R$, then
    \[\E K_{n,d_n}^{(r)}\to
      \frac{\e^{\frac 12-c}}
      {\sqrt{2\pi}}\cdot\frac{\Gamma(r+1-\e)}{r!},\qquad r\ge2,\]
    where $\Gamma$ stands for the Gamma function.
  \item If $d_n-(\e\log n-\frac12\log\log n)\to\infty$, then
    $\E K_{n,d_n}^{(r)}\to0$ for any $r\ge2$.
  \end{enumerate}
\end{theorem}

The assumption $d_n\neq 1$ in (i) is needed to exclude the
one-dimensional case, in which $K_{n,1}^{(r)}=1$.

A similar picture arises in a related model in which the points $X_i$ are the atoms of a homogeneous Poisson point measure on $[0,1]^d$ with intensity $\lambda$. All of the above results for non-Pareto points and for points dominating exactly $r$ others remain valid with $n$ replaced throughout by $\lambda$. This can be easily derived from Theorems~\ref{thm:1}\hs--\hs\ref{thm:3} by a standard poissonization argument.

\section{Proof of Theorems \ref{thm:1} and \ref{thm:2}}

The overall idea of the proof is as follows. Since the points are i.i.d.\ uniform in $[0,1]^d$, we may assume without loss of generality that in \eqref{eq:pi} we have $\{k_1,\ldots,k_m\}=\{1,\ldots,m\}$. For a fixed Borel set $U\subset[0,1]^m$, define
\begin{equation}
\label{eq:DS}
D_{n,d}(i)=\sum_{\substack{1\le j\le n\\ j\ne i}}\one\{X_j\preceq X_i\},\qquad
S_{n,d}(U)=\sum_{i=1}^n D_{n,d}(i)\one\{\pi_m X_i\in U\}.
\end{equation}
Thus, $S_{n,d}(U)$ counts all non-Pareto points whose projection lies in $U$, with each point weighted by the number of other points it dominates. Meanwhile, the unweighted number of such points is given by
\begin{equation}
\label{eq:T}
T_{n,d}(U)=\sum_{i=1}^n \one\{D_{n,d}(i)\ge1\}\one\{\pi_m X_i\in U\}.
\end{equation}
Although $T_{n,d}(U)\le S_{n,d}(U)$, we clearly have
\begin{equation}
\label{eq:ST}
\mathbb P\{T_{n,d}(U)=0\}=\mathbb P\{S_{n,d}(U)=0\}.
\end{equation}
The indicators in \eqref{eq:DS} have a noticeably simpler dependence structure than those in \eqref{eq:T}. This allows us to apply the Stein--Chen Poisson approximation in the Arratia--Goldstein--Gordon form \cite{AGG89} (see also \cite{AGG90}) to show that, in the regime
$d_n=\frac{2}{\log 2}\log n+c+\smallO(1)$, $S_{n,d_n}(U)$ converges in distribution to a Poisson random variable. By \eqref{eq:ST}, we obtain convergence of the void probabilities for $T_{n,d_n}(U)$.

The asymptotics of $\E T_{n,d_n}(U)$ can be computed
explicitly. Since, by Kallenberg's criterion, $vd$-convergence of
random point measures is implied by convergence of their means and
void probabilities, Theorem~\ref{thm:2} follows. Setting $U=[0,1]^m$,
we obtain the distributional convergence in
Theorem~\ref{thm:1}\hs(ii); the remaining claims in (ii) follow by
analytic arguments. Finally, (i) is proved using (ii) and a coupling
of the points in different dimensions, and (iii) follows from an
analytic argument.

We now turn to the detailed proof. We will need several lemmas. Recall
that $m\in\mathbb N$ and a Borel subset $U$ of $[0,1]^m$ are fixed,
and $|x|$ denotes the product of the coordinates of $x$. Furthermore,
let $x\wedge y$ denote the component-wise minimum of the vectors $x$
and $y$.

\begin{lemma}
  \label{le:1}
  If $d_n=\frac{2}{\log 2}\log n+c+\smallO(1)$, then
  \[\E T_{n,d_n}(U)\to 2^{m-c}\int_U|x|\d x,\qquad n\to\infty.\]
\end{lemma}

\begin{proof}
For each $i$,
\begin{equation}
\label{eq:PP}
\begin{aligned}
  \P\{X_i\text{ is Pareto},\pi_mX_i\in U\}
  &=\E\Bigl(\P\Bigl(\bigcap_{j\ne i}\{X_j\npreceq X_i\}\hs\mid\hs X_i\Bigr)\hs\one\{\pi_mX_i\in U\}\Bigr)\\
  &=\E\Bigl((1-|X_i|)^{n-1}\hs\one\{\pi_mX_i\in U\}\Bigr)\\
    &=\int_{U\times[0,1]^{d_n-m}}(1-|x|)^{n-1}\d x.
\end{aligned}
\end{equation}
Hence,
\begin{equation}
\begin{aligned}
\label{eq:PP'}
\E T_{n,d_n}(U)&=\sum_{i=1}^n\P\{X_i\text{ is non-Pareto},\pi_mX_i\in U\}\\&=n\int_{U\times[0,1]^{d_n-m}}\bigl(1-(1-|x|)^{n-1}\bigr)\d x.
\end{aligned}
\end{equation}
Since
\begin{equation}
  \label{eq:ie_1}
  (n-1)|x|-\frac{(n-1)(n-2)}2|x|^2\le 1-(1-|x|)^{n-1}\le(n-1)|x|,
\end{equation}
we have
\begin{equation}
\begin{aligned}
\label{eq:ET}
n(n-1)\int_{U\times[0,1]^{d_n-m}}|x|\d x
&-\O\bigl(n^3\bigr)\cdot\int_{U\times[0,1]^{d_n-m}}|x|^2\d x\\&\le\E T_{n,d_n}(U)\le n(n-1)\int_{U\times[0,1]^{d_n-m}}|x|\d x.
\end{aligned}
\end{equation}
The integrals are straightforward to evaluate:
\begin{gather}
\label{eq:int}
\int_{U\times[0,1]^{d_n-m}}|x|\d x=2^{m-d_n}\int_U|x|
\d x=2^{m-c}n^{-2}\bigl(1+\smallO(1)\bigr)\cdot
\int_U|x|\d x,\\
\label{eq:int'}
\int_{U\times[0,1]^{d_n-m}}|x|^2\d x=3^{m-d_n}\int_U|x|^2\d x=\O\Bigl(n^{-\frac{2\log 3}{\log 2}}\Bigr)=\smallO\bigl(n^{-3}\bigr).
\end{gather}
Thus, \eqref{eq:ET} yields $\lim_{n\to\infty}\E T_{n,d_n}(U)=2^{m-c}\int_U|x|\d x$. This completes the proof.
\end{proof}

\begin{lemma}
  \label{le:2}
  If $d_n=\frac{2}{\log 2}\log n+c+\smallO(1)$, then
  \[\Var(n-K_{n,d_n})\to 2^{-c},\qquad n\to\infty.\]
\end{lemma}

\begin{proof}
  By Lemma~\ref{le:1}, we have
  $\E(n-K_{n,d_n})=\E T_{n,d_n}\bigl([0,1]^m\bigr)\to 2^{-c}$.
  As
  \begin{displaymath}
    \Var\xi=\E \xi(\xi-1)+\E \xi-(\E \xi)^2,
  \end{displaymath}
  it remains to show that
  \[\E\bigl((n-K_{n,d_n})(n-K_{n,d_n}-1)\bigr)\to 2^{-2c}.\]	
  Write
  $n-K_{n,d_n}=\sum_{i=1}^n\one\{X_i\text{ is non-Pareto}\}$.
  Then
  \begin{equation}
    \label{eq:fm}
    \begin{aligned}
      \E\bigl((n-K_{n,d_n}&)(n-K_{n,d_n}-1)\bigr)\\&=
      \sum_{i\ne j}\E\bigl(\one\{X_i\text{ is non-Pareto}\}
      \one\{X_j\text{ is non-Pareto}\}\bigr)\\&=
      n(n-1)\hs\P\{\text{both }X_1\text{ and }X_2\text{ are non-Pareto}\}.
    \end{aligned}
  \end{equation}
  For $x\in[0,1]^{d_n}$, let $R_x=\sum_{i=3}^n\one\{X_i\preceq x\}$,
  and set
  \[D=\bigl\{(x_1,x_2)\in[0,1]^{d_n}\times[0,1]^{d_n}\colon x_1\npreceq x_2,
    x_2\npreceq x_1\bigr\}.\]
  Then
  \begin{equation}
    \label{eq:2NP}
    \begin{aligned}
      \P\{X_1\text{ and }X_2\text{ are non-Pareto}\}
      =\P\bigl\{(X_1,X_2)\in D,R_{X_1}\ge1,R_{X_2}\ge1&\bigr\}
      \\+\P\{X_1\preceq X_2,R_{X_1}\ge1\}+\P\{X_2\preceq X_1,R_{X_2}\ge1&\}.
    \end{aligned}
  \end{equation}
  By the union bound,
  \begin{equation}
  	\begin{aligned}
    \label{eq:term2}
    \P\{X_1\preceq X_2,R_{X_1}\ge1\}
    &\le (n-2)\hs\P\{X_3\preceq X_1\preceq X_2\}\\&=
    (n-2)\,6^{-d_n}\sim 6^{-c}n^{1-2\frac{\log 6}{\log 2}}=\smallO\bigl(n^{-2}\bigr),
    \end{aligned}
  \end{equation}
  and the same applies to $\P\{X_2\preceq X_1,R_{X_2}\ge1\}$.
  Thus, by \eqref{eq:fm}\hs--\hs\eqref{eq:term2} and conditioning on $(X_1,X_2)$,
  \begin{equation}
  	\begin{multlined}
    \label{eq:fm2}
    \E\bigl((n-K_{n,d_n})(n-K_{n,d_n}-1)\bigr)\\=
    n(n-1)\int_D\P\{R_{x_1}\ge1,R_{x_2}\ge1\}\d x_1\d x_2+\smallO(1).
    \end{multlined}
  \end{equation}
  
  For our purposes, the following rough bound will suffice:
  \[0\le r_1r_2-\one\{r_1\ge1\}\one\{r_2\ge1\}\le r_1\binom {r_2}2+
    r_2\binom {r_1}2,\qquad r_1,r_2\in\N_0.\]
  This inequality is trivial if $r_1=0$ or $r_2=0$, and easily verified for $r_1,r_2\ge1$. Hence,
  \begin{equation}
    \label{eq:rs}
    \bigl|\P\{R_{x_1}\ge1,R_{x_2}\ge1\}-\E(R_{x_1}R_{x_2})\bigr|
    \le\E\biggl(R_{x_1}\binom{R_{x_2}}2\biggr)+
    \E\biggl(R_{x_2}\binom{R_{x_1}}2\biggr).
  \end{equation}
  Since $X_3,\ldots,X_n$ are i.i.d.\ uniform on $[0,1]^{d_n}$, a straightforward counting argument yields
  \begin{gather*}
    \E(R_{x_1}R_{x_2})=(n-2)|x_1\wedge x_2|+(n-2)(n-3)|x_1||x_2|,\\
    \E\biggl(R_{x_1}\binom{R_{x_2}}2\biggr)
    =(n-2)(n-3)|x_2||x_1\wedge x_2|+\frac{(n-2)(n-3)(n-4)}2|x_1||x_2|^2,\\
    \E\biggl(R_{x_2}\binom{R_{x_1}}2\biggr)
    =(n-2)(n-3)|x_1||x_1\wedge x_2|+\frac{(n-2)(n-3)(n-4)}2|x_1|^2|x_2|.
  \end{gather*}
  For instance, the last formula is obtained by splitting according to
  whether the point counted by $R_{x_2}$ is one of the two points
  counted by $\binom{R_{x_1}}{2}$ or not. The former case yields the
  first term on the right-hand side, while the latter yields the second.
  
  Therefore, integrating \eqref{eq:rs} over $D$, we get
  \begin{equation}
    \label{eq:mainD}
    \begin{aligned}
      \int_D&\P\{R_{x_1}\ge1,R_{x_2}\ge1\}\d x_1\d x_2
      \\&=(n-2)\int_D|x_1\wedge x_2|\d x_1\d x_2
      +(n-2)(n-3)\int_D|x_1||x_2|\d x_1\d x_2
      \\&+\O\Bigl(n^2\!\!\int_{[0,1]^{2d_n}}\!\!|x_1||x_1\wedge x_2|\d x_1\d x_2+n^3\!\!\int_{[0,1]^{2d_n}}\!\!|x_1|^2|x_2|\d x_1\d x_2\Bigr).
    \end{aligned}
  \end{equation}
  
  The integrals are straightforward to evaluate by probabilistic arguments:
  \begin{gather*}
    \int_D|x_1\wedge x_2|\d x_1\d x_2
    =\P\bigl\{(X_1,X_2)\in D,\hs X_3\preceq X_1,\hs X_3\preceq X_2\bigr\}
    =3^{-d_n}-2\cdot 6^{-d_n}=\smallO\bigl(n^{-3}\bigr),\\
    \int_D|x_1||x_2|\d x_1\d x_2
    =\P\bigl\{(X_1,X_2)\in D,\hs X_3\preceq X_1,\hs X_4\preceq X_2\bigr\}
    =4^{-d_n}-2\cdot 8^{-d_n}\sim 2^{-2c}n^{-4},\\
    \int_{[0,1]^{2d_n}}|x_1||x_1\wedge x_2|\d x_1\d x_2
    =\P\{X_3\preceq X_1,\hs X_4\preceq X_1,\hs X_4\preceq X_2\}=(24/5)^{-d_n}=\smallO\bigl(n^{-4}\bigr),\\
    \int_{[0,1]^{2d_n}}|x_1|^2|x_2|\d x_1\d x_2=
    \P\{X_3\preceq X_1,\hs X_4\preceq X_1,\hs X_5\preceq X_2\}
    =6^{-d_n}=\smallO\bigl(n^{-5}\bigr).
  \end{gather*}
  Hence, by \eqref{eq:mainD}, $\int_D\P\{R_{x_1}\ge1,R_{x_2}\ge1\}\d x_1\d x_2\sim2^{-2c}n^{-2}$, and \eqref{eq:fm2} yields the claim.
\end{proof}

In principle, the same approach could be extended to higher moments, which, combined with the method of moments, would yield a Poisson approximation. However, the required arguments become increasingly technical as the order of the moment grows. We therefore prefer using the more elegant Stein--Chen approach based on dependency neighborhoods.	

\begin{lemma}
	\label{le:3}
	If $d_n=\frac{2}{\log 2}\log n+c+\smallO(1)$, then
	\[\P\bigl\{S_{n,d_n}(U)=0\bigr\}\to
	\exp\bigl(-2^{m-c}\int_U |x|\d x\bigr),\qquad n\to\infty.\]
\end{lemma}

\begin{proof}
	Denote $[n]=\{1,\ldots,n\}$ and
	$[n]_{\ne}^2=\bigl\{(i,j)\in[n]^2\colon i\ne j\bigr\}$.
	It follows from \eqref{eq:DS} that
	\[S_{n,d_n}(U)=\sum_{(i,j)\in[n]_{\ne}^2}
	\one\{X_j\preceq X_i,\pi_m X_i\in U\}.\]
	Let $A_{ij}=\{X_j\preceq X_i,\pi_m X_i\in U\}$ for $(i,j)\in[n]_{\ne}^2$. Observe that $A_{ij}$ and $A_{i'j'}$ are independent whenever $\{i,j\}\cap\{i',j'\}=\emptyset$ and define the dependency neighborhood of $(i,j)$ by
	\begin{align*}
	\mathcal N_{ij}&=\bigl\{(i',j')\in[n]_{\ne}^2\colon\{i,j\}\cap\{i',j'\}
	\ne\emptyset\bigr\}
	\\&=\bigl\{(i,j)\bigr\}\cup\bigl\{(j,i)\bigr\}
	\cup\!\!\bigcup_{k\in[n]\setminus\{i,j\}}\!\!
	\bigl\{(i,k),(j,k),(k,i),(k,j)\bigr\}.
	\end{align*}
	Thus, $\mathcal N_{ij}$ contains $4n-6$ pairs.
	
	The Poisson approximation in the form of the Arratia--Goldstein--Gordon bound (see Theorem~1 in \cite{AGG89}) yields
	\begin{equation}
	\label{eq:AGG}
	\begin{multlined}
	\Bigl|\P\bigl\{S_{n,d_n}(U)=0\bigr\}-
	\exp\Bigl(-\!\!\!\sum_{(i,j)\in[n]_{\ne}^2}
	\!\!\!\P(A_{ij})\Bigr)\Bigr|
	\le\sum_{(i,j)\in[n]_{\ne}^2}
	\sum_{(i',j')\in\mathcal N_{ij}}\P(A_{ij})\P(A_{i'j'})\\+
	\sum_{(i,j)\in[n]_{\ne}^2}
	\sum_{(i',j')\in\mathcal N_{ij}\setminus\{(i,j)\}}\P(A_{ij}\cap A_{i'j'}).
	\end{multlined}
	\end{equation}
	Since by \eqref{eq:int},
	\begin{equation}
	\label{eq:int2}
	\P(A_{ij})=
	\int_{U\times [0,1]^{d_n-m}}|x|\d x\sim 2^{m-c}n^{-2}\cdot
	\int_U|x|\d x,\qquad n\to\infty,
	\end{equation}
	the first double sum on the right-hand side of \eqref{eq:AGG}
        is of the order
        \[n(n-1)(4n-6)\mathcal O\bigl(n^{-4}\bigr)\to0.\]
    Furthermore, since
	$\P(A_{ij}\cap A_{ji})=0$ for $i\ne j$, and, for pairwise distinct indices,
	\begin{gather*}
	\P(A_{ij}\cap A_{i'j})
	\le\P\{X_j\preceq X_i,\hs X_j\preceq X_{i'}\}=3^{-d_n}=\smallO\bigl(n^{-3}\bigr),
	\\\P(A_{ij}\cap A_{ij'})
	\le\P\{X_j\preceq X_i,\hs X_{j'}\preceq X_i\}=3^{-d_n}=\smallO\bigl(n^{-3}\bigr),
	\\\P(A_{ij}\cap A_{ji'})
	\le\P\{X_{i'}\preceq X_j\preceq X_i\}=6^{-d_n}=\smallO\bigl(n^{-5}\bigr),
	\\\P(A_{ij}\cap A_{j'i})
	\le\P\{X_j\preceq X_i\preceq X_{j'}\}
	=6^{-d_n}=\smallO\bigl(n^{-5}\bigr),
	\end{gather*}
	the second double sum on the right-hand side of \eqref{eq:AGG}
        is of the order
        \[n(n-1)(4n-8)\smallO\bigl(n^{-3}\bigr)\to0.\]
	Finally, by \eqref{eq:int2}, the exponential term on the left-hand side of \eqref{eq:AGG} converges to
	$\exp\bigl(-2^{m-c}\int_U |x|\d x\bigr)$ as $n\to\infty$. Hence, \eqref{eq:AGG} yields the claim.
\end{proof}

Note that the Arratia--Goldstein--Gordon bound in fact yields the convergence of $S_{n,d_n}(U)$ to the corresponding Poisson random variable in total variation distance. For our purposes, the established convergence of the void probabilities is sufficient.

\begin{proof}[Proof of Theorem \ref{thm:2}]
By Kallenberg's criterion (see, e.g., Theorem~4.18 in \cite{K17}), it suffices to show that
\[\E\hs\Xi_n^m(U)\to \E\hs\Xi^m(U)
\quad\text{and}\quad
\P\{\Xi_n^m(U)=0\}\to\P\{\Xi^m(U)=0\},\qquad n\to\infty,\]
for every $U\in\B(\R^m)$. Since $\Xi_n^m(U)=T_{n,d_n}(U)$ and
$\Xi^m(U)\sim\mathsf{Pois}\bigl(2^{m-c}\int_U|x|\d x\bigr)$,
the former convergence follows from Lemma~\ref{le:1}, while the latter follows from Lemma~\ref{le:3} and \eqref{eq:ST}.
\end{proof}

\begin{proof}[Proof of Theorem \ref{thm:1}]
The distributional convergence in (ii) follows from Theorem~\ref{thm:2}, while the convergence of expectations and variances follows from Lemmas~\ref{le:1} and \ref{le:2}. To prove that $\sum_{r=2}^{n-1} K_{n,d_n}^{(r)}\xrightarrow{p}0$, it is enough to show that
\begin{equation}
\label{eq:EK1}
\lim_{n\to\infty}\E K_{n,d_n}^{(1)}=2^{-c};
\end{equation}
indeed, combined with the already proved convergence
$\E(n-K_{n,d_n})\to 2^{-c}$, this implies
$\E\sum_{r=2}^{n-1} K_{n,d_n}^{(r)}\to0,$
and thus, by non-negativity, the convergence in probability.

Arguing as in \eqref{eq:PP}, we obtain
\[\E K_{n,d_n}^{(1)}=n\P\{X_1\text{ dominates exactly one other point}\}
=n(n-1)\int_{[0,1]^{d_n}}|x|(1-|x|)^{n-2}\d x.\]
Hence, by \eqref{eq:ie_1},
\begin{equation*}
\E K_{n,d_n}^{(1)}\ge n(n-1)\int_{[0,1]^{d_n}}|x|\d x-n(n-1)(n-2)\int_{[0,1]^{d_n}}|x|^2\d x,
\end{equation*}
and thus, by \eqref{eq:int} and \eqref{eq:int'}, $\liminf_{n\to\infty}\E K_{n,d_n}^{(1)}\ge 2^{-c}$.
Since
\[\E K_{n,d_n}^{(1)}\le\E(n-K_{n,d_n})\to 2^{-c},\]
\eqref{eq:EK1} follows.

To prove (i), we combine (ii) with a coupling argument. Let $\bigl(X_i^k\bigr)_{i,k\in\N}$ be independent random variables uniformly distributed on $[0,1]$. Then, for fixed $n$ and $d$, the points $X_i=\bigl(X_i^1,\ldots,X_i^d\bigr)$, $i=1,\ldots,n$, are independent and uniformly distributed in $[0,1]^d$. Moreover, $n-K_{n,d'}\le n-K_{n,d}$ for $d'\ge d$, since under this coupling any non-Pareto point in higher dimension is also non-Pareto in lower dimension.

Let $d_n-\frac 2{\log 2}\log n\to-\infty$. Fix $c\in\R$ and consider a sequence $\bigl(d_n^{\langle c\rangle}\bigr)_{n\in\N}$ such that $d_n^{\langle c\rangle}\ge d_n$ for all $n$ and $d_n^{\langle c\rangle}-\frac{2}{\log 2}\log n\to c$. Then
$n-K_{n,d_n}\ge n-K_{n,d_n^{\langle c\rangle}}$, and hence, for any $N\in\N$,
\[\liminf_{n\to\infty}\P\{n-K_{n,d_n}\ge N\}\ge
\lim_{n\to\infty}\P\bigl\{n-K_{n,d_n^{\langle c\rangle}}\ge N\bigr\}
=\P\{Z_c\ge N\},\]
where $Z_c\sim\mathsf{Pois}(2^{-c})$. Letting now $c\to -\infty$, we
obtain $\lim_{n\to\infty}\P\{n-K_{n,d_n}\ge N\}=1$, which implies
$n-K_{n,d_n}\xrightarrow{p}\infty$ and proves (i).

Finally, let $d_n-\frac{2}{\log 2}\log n\to\infty$. Then (iii) follows
from
\begin{equation}
  \label{eq:thm:1_iii}
  \E(n-K_{n,d_n})=\E T_{n,d_n}\big([0,1]^m\big)\leq\E
  S_{n,d_n}\big([0,1]^m\big)=n(n-1)2^{-d_n}\to0.\qedhere
\end{equation}
\end{proof}

\section{Proof of Theorem \ref{thm:3}}

Fix an integer $r\ge 2$. Arguing as in \eqref{eq:PP} and \eqref{eq:PP'}, we obtain
\begin{align*}
	\E K_{n,d_n}^{(r)}&=n\hs\E\Bigl(\P\Bigl(\bigcap_{j\in\sR}
	\{X_j\preceq X_1\}\cap\hspace{-14pt}\bigcap_{j\in[n]\setminus(\{1\}\cup\sR)}\hspace{-14pt}\{X_j\npreceq X_1\}\\
	&\hspace{5cm}\text{ for some $\sR\subset[n]\setminus\{1\}$ of cardinality $r$}\hs\mid\hs X_1\Bigr)\Bigr)\\
	&=n\binom{n-1}r\E\bigl(|X_1|^r(1-|X_1|)^{n-1-r}\bigr),\qquad n\ge r+1.
\end{align*}

Since $|X_1|$ is the product of the $d_n$ i.i.d.\ $\mathsf{Unif}([0,1])$ coordinates of $X_1$, we may write $|X_1|=\e^{-G}$ with $G\sim\mathsf{Gamma}(d_n,1)$. Consequently, $|X_1|$ has density
\[f_{|X_1|}(x)=\frac{(-\log x)^{d_n-1}}{(d_n-1)!},\qquad x\in(0,1].\]
Therefore, as $n\to\infty$,
\begin{equation}
\begin{aligned}
	\label{eq:np_r}
	\E K_{n,d_n}^{(r)}
	&\sim\frac{n^{r+1}}{r!\hs(d_n-1)!}
	\int_0^1 x^r(1-x)^{n-1-r}(-\log x)^{d_n-1}\d x\\
	&=\frac 1{r!\hs(d_n-1)!}\int_0^n x^r\Bigl(1-\frac xn\Bigr)^{n-1-r}
	\Bigl(\log \frac nx\Bigr)^{d_n-1}\d x\\
	&=\frac {(\log n)^{d_n-1}}{r!\hs(d_n-1)!}\int_0^n x^r\Bigl(1-\frac xn\Bigr)^{n-1-r}
	\Bigl(1-\frac{\log x}{\log n}\Bigr)^{d_n-1}\d x.
\end{aligned}
\end{equation}

The remainder of the proof is mainly analytic and consists of an
asymptotic analysis of this integral by Laplace's method in the regime
$d_n=\e\log n-\frac12\log\log n+c+\smallO(1)$. We will need several
lemmas.

\begin{lemma}
	\label{le:np_r_1}
	If $d_n=\e\log n-\frac 12\log\log n+c+\smallO(1)$, then
	\[\frac{(\log n)^{d_n-1}}{(d_n-1)!}\to
	\frac{\e^{\frac 12-c}}{\sqrt{2\pi}},\qquad n\to\infty.\]
\end{lemma}

\begin{proof}
  By Stirling's formula, as $n\to\infty$,
  \begin{equation}
    \begin{aligned}
      \label{eq:St}
      \frac{(\log n)^{d_n-1}}{(d_n-1)!}&\sim
      \frac{(\log n)^{d_n-1}\e^{d_n-1}}{\sqrt{2\pi (d_n-1)}(d_n-1)^{d_n-1}}\\
      &=\frac 1{\sqrt{2\pi (d_n-1)}}
      \exp\bigl((d_n-1)(1+\log \log n-\log (d_n-1))\bigr).
    \end{aligned}
  \end{equation}
  Denoting $a_n=-\frac 12\log\log n+c-1+\smallO(1)$, we have
  \begin{multline*}
  \log (d_n-1)=\log(\e\log n+a_n)\\=\log(\e\log n)+ \log\Bigl(1+\frac{a_n}{\e\log n}\Bigr)
    =1+\log\log n+\frac{a_n}{\e\log n}+\delta_n,
  \end{multline*}
  where $\delta_n=\O\bigl(\frac{\log\log n}{\log n}\bigr)^2=\smallO\bigl(\frac 1{\log n}\bigr)$.
  Hence,
  \[(d_n-1)(1+\log \log n-\log (d_n-1))=-(\e\log n+a_n)\Bigl(\frac{a_n}{\e\log n}+\delta_n\Bigr)=-a_n+\smallO(1).\]
  Thus, \eqref{eq:St} yields
  \[\frac{(\log n)^{d_n-1}}{(d_n-1)!}
    \sim\frac 1{\sqrt{2\pi\e\log n}}\exp\Bigl(\frac 12
    \log\log n-c+1+\smallO(1)\Bigr)
    \to\frac{\e^{1-c}}{\sqrt{2\pi\e}}.\qedhere\]
\end{proof}

\begin{remark}
  \label{re:np_r_1}
  Lemma \ref{le:np_r_1} also implies the following two one-sided
  results.
  \begin{enumerate}[(i)]
    \item If $d_n\geq 2$ for large $n$ and $d_n-(\e\log
      n-\frac{1}{2}\log\log n)\to -\infty$, then
      \begin{displaymath}
        \frac{(\log n)^{d_n-1}}{(d_n-1)!}\to\infty.
      \end{displaymath}
    \item If $d_n-(\e\log
      n-\frac{1}{2}\log\log n)\to \infty$, then
      \begin{displaymath}
        \frac{(\log n)^{d_n-1}}{(d_n-1)!}\to 0.
      \end{displaymath}
    \end{enumerate}
    This follows from Lemma \ref{le:np_r_1} and the fact that, for
    fixed $n$, $\frac{(\log n)^{k-1}}{(k-1)!}$ is increasing for
    $k<\log n$ and decreasing for $k>\log n$.
\end{remark}

\begin{lemma}
  \label{le:np_r_2}
  If $d_n=\e\log n-\frac 12\log\log n+c+\smallO(1)$, then
  \begin{equation}
    \label{eq:rest}
    \frac {(\log n)^{d_n-1}}{(d_n-1)!}\int_{\log n}^n
    x^r\Bigl(1-\frac xn\Bigr)^{n-1-r}
    \Bigl(1-\frac{\log x}{\log n}\Bigr)^{d_n-1}\d x\to0,\qquad n\to\infty.
  \end{equation}	
\end{lemma}

\begin{proof}
  For large $n$ and $x\in[0,n]$, the inequality $1-t\le\e^{-t}$ implies
  \[\Bigl(1-\frac xn\Bigr)^{n-1-r}\le
    \exp\Bigl(-\frac{(n-1-r)x}n\Bigr)\le\e^{-x/2}.\]
  Hence, for such $n$, the left-hand side of \eqref{eq:rest} is bounded by
  \[\frac{(\log n)^{d_n-1}}{(d_n-1)!}\int_{\log n}^nx^r\e^{-x/2}\d x,\]
  which vanishes as $n\to\infty$ in view of Lemma~\ref{le:np_r_1}.
\end{proof}

\begin{lemma}
  \label{le:np_r_3}
  If $d_n=\e\log n-\frac 12\log \log n+c+\smallO(1)$, then
  \begin{equation*}
    \int_0^{\log n}
    x^r\Bigl(1-\frac xn\Bigr)^{n-1-r}
    \Bigl(1-\frac{\log x}{\log n}\Bigr)^{d_n-1}\d x\to \Gamma(r+1-\e),\qquad n\to\infty.
  \end{equation*}
\end{lemma}

\begin{proof}
  First, note that for large $n$,
  $x\in\bigl(0,(\log n)^{-1}\bigr]$, and any $\eps>0$,
  \begin{displaymath}
    \Bigl(1-\frac{\log x}{\log n}\Bigr)^{d_n-1}\le\exp\Bigl(-(d_n-1)\frac{\log x}{\log n}\Bigr)
    \le x^{-\e-\eps}.
  \end{displaymath}
  Hence, the integral over $\bigl[0,(\log n)^{-1}\bigr]$ is bounded by
  $\int_0^{(\log n)^{-1}}x^{r-\e-\eps}\d x$,
  which vanishes for small $\eps$ as $n\to\infty$. Therefore, it suffices to consider the integral over $\bigl[(\log n)^{-1},\log n\bigr]$.
  
  Next, observe that
  \begin{equation}
  \begin{gathered}
    \label{eq:equiv}
    \sup_{(\log n)^{-1}\le x\le\log n}
    \biggl|\frac{\bigl(1-\frac xn\bigr)^{n-1-r}}{\e^{-x}}-1\biggr|\to 0,\\
    \sup_{(\log n)^{-1}\le x\le\log n}\Biggl|
    \frac{\bigl(1-\frac{\log x}
      {\log n}\bigr)^{d_n-1}}{x^{-\e}}-1\Biggr|\to 0
  \end{gathered}
  \end{equation}
  as $n\to\infty$. Indeed, for any sequence $(x_n)_{n\in\N}$ with $x_n\in\bigl[(\log n)^{-1},\log n\bigr]$, we have
  \begin{gather*}
  	\begin{multlined}
    (n-1-r)\log\Bigl(1-\frac{x_n}{n}\Bigr)+x_n\\
    =(n-1-r)\biggl(-\frac{x_n}{n}+\O\biggl(\frac{x_n^2}{n^2}\biggr)
    \biggr)+x_n
    =\O\Bigl(\frac{x_n}n\Bigr)+\O\biggl(\frac{x_n^2}n\biggr)\to0,
    \end{multlined}\\
	\begin{multlined}
    (d_n-1)\log\Bigl(1-\frac{\log x_n}{\log n}\Bigr)+\e\log x_n\\=
    (d_n-1)\Bigl(-\frac{\log x_n}{\log n}+\O\Bigl(\frac{\log x_n}{\log n}\Bigr)^{\!2\,}\Bigr)+\e\log x_n\to0.
    \end{multlined}
  \end{gather*}
  Since this holds for every such sequence, the claimed uniform convergence follows.
  
  It is easy to check that, if, for some positive functions $f_n$, $h_n$, $h$, and sets $A_n$,
  \[\sup_{x\in A_n}\biggl|\frac{h_n(x)}{h(x)}-1\biggr|\to0,\qquad n\to\infty,\]
  then
  \[\lim_{n\to\infty}\int_{A_n}f_n(x)h_n(x)\d x=\lim_{n\to\infty}
    \int_{A_n}f_n(x)h(x)\d x\]
  whenever the second limit exists.
  Applying this to the integral
  \[\int_{(\log n)^{-1}}^{\log n}
    x^r\Bigl(1-\frac xn\Bigr)^{n-1-r}
    \Bigl(1-\frac{\log x}{\log n}\Bigr)^{d_n-1}\d x\]
  and using \eqref{eq:equiv} together with
  \[\lim_{n\to\infty}\int_{(\log n)^{-1}}^{\log n}
    x^r\e^{-x}x^{-\e}\d x=\Gamma(r+1-\e),\]
  we obtain the claim.
\end{proof}

\begin{proof}[Proof of Theorem \ref{thm:3}]
  (ii) follows from \eqref{eq:np_r} and
  Lemmas~\ref{le:np_r_1}, \ref{le:np_r_2}, \ref{le:np_r_3}. For (i),
  it suffices to note that, by Remark \ref{re:np_r_1}(i), the first
  factor on the right-hand side of \eqref{eq:np_r} tends to infinity,
  whereas the integral is bounded away from zero as $n\to\infty$:
  \begin{align*}
    \int_{0}^{n}x^r
    \Big(1-\frac{x}{n}\Big)^{n-1-r}\Big(1-\frac{\log x}{\log
      n}\Big)^{d_n-1}\d x
    &\geq\int_{1}^{2}x^r
    \Big(1-\frac{x}{n}\Big)^{n-1-r}\Big(1-\frac{\log x}{\log
      n}\Big)^{d_n-1}\d x\\
    &\geq \Big(1-\frac{2}{n}\Big)^{n-1-r}\Big(1-\frac{\log 2}{\log
      n}\Big)^{d_n-1}\int_{1}^{2}x^r\d x,
  \end{align*}
  where each factor on the right-hand side is bounded away from zero.

  We now prove (iii). Since $2.9>\frac{2}{\log 2}$, it suffices to
  consider the case where $d_n\leq 2.9\log n$ for large $n$; indeed,
  for any subsequence $(n_k)$ such that $d_{n_k}>2.9\log n_k$,
  \begin{displaymath}
    \E K_{n_k,d_{n_k}}^{(r)}\leq \E (n_k-K_{n_k,d_{n_k}})\to 0
  \end{displaymath}
  by \eqref{eq:thm:1_iii}. Moreover, by \eqref{eq:np_r} and Remark
  \ref{re:np_r_1}(ii), the proof of (iii) reduces to showing that
  \begin{equation}
    \label{eq:np_iii_rest}
    \limsup_{n\to\infty}\int_{0}^{n}x^r
    \Big(1-\frac{x}{n}\Big)^{n-1-r}\Big(1-\frac{\log x}{\log
      n}\Big)^{d_n-1}\d x<\infty.
  \end{equation}
  The integral over $[1,n]$ is decreasing in $d_n$. As shown in (ii),
  \eqref{eq:np_iii_rest} holds for $d_n=\e\log n-\frac{1}{2}\log\log
  n+c+\smallO(1)$. Hence, provided that the integral is restricted to
  $[1,n]$, the same holds for larger values of $d_n$. In view of the
  above, it remains to prove that
  \begin{displaymath}
    \limsup_{n\to\infty}\int_{0}^{1}x^r
    \Big(1-\frac{x}{n}\Big)^{n-1-r}\Big(1-\frac{\log x}{\log
      n}\Big)^{d_n-1}\d x<\infty
  \end{displaymath}
  for $d_n\leq 2.9\log n$. By monotonicity, it suffices to take
  $d_n=2.9\log n$. The claim then follows by dominated convergence,
  since
  \begin{displaymath}
    x^r\Big(1-\frac{x}{n}\Big)^{n-1-r}\Big(1-\frac{\log x}{\log
      n}\Big)^{2.9\log n-1}\to\e^{-x}x^{r-2.9},\quad x\in(0,1],
  \end{displaymath}
  and, for $n\geq r+1$, the integrand is dominated on $(0,1]$ by
  $x^{r-2.9}$, which is integrable since $r\geq 2$.
\end{proof}

\section*{Acknowledgments}
  We are grateful to Ilya Molchanov and Chinmoy Bhattacharjee
  for bringing this problem to our attention, as well as for helpful
  and stimulating discussions.

\providecommand{\bysame}{\leavevmode\hbox to3em{\hrulefill}\thinspace}
\providecommand{\MR}{\relax\ifhmode\unskip\space\fi MR }
\providecommand{\MRhref}[2]{%
	\href{http://www.ams.org/mathscinet-getitem?mr=#1}{#2}}
\providecommand{\href}[2]{#2}

\end{document}